\documentclass[12pt]{amsart}
\usepackage[myheadings]{fullpage}

\usepackage{amsmath,amssymb,amsthm}

\numberwithin{equation}{section}

\newtheorem{theorem}{Theorem}

\newtheorem{lemma}{Lemma}[section]
\newtheorem{prop}[lemma]{Proposition}

\theoremstyle{definition}
\newtheorem{question}[lemma]{Question}

\theoremstyle{remark}
\newtheorem{remark}[lemma]{Remark}
\newtheorem{example}[lemma]{Example}

\DeclareMathOperator{\id}{id}
\DeclareMathOperator{\vol}{vol}
\DeclareMathOperator{\dist}{dist}
\DeclareMathOperator{\trace}{trace}
\DeclareMathOperator{\diam}{diam}

\newcommand{\R}{\mathbb R}

\renewcommand{\phi}{\varphi}
\newcommand{\ep}{\varepsilon}
\newcommand{\pd}{\partial}

\newcommand{\be}{\begin{equation}}
\newcommand{\ee}{\end{equation}}

\begin{document}

\title{On Gromov--Hausdorff stability in a boundary rigidity problem}

\author{Sergei Ivanov}
\thanks{Supported by the Dynasty foundation and RFBR grant 09-01-12130-ofi-m}
\email{svivanov@pdmi.ras.ru}
\address{Saint Petersburg department of Steklov Math Institute,
191023, Fontanka 27, Saint Petersburg, Russia}

\begin{abstract}
Let $M$ be a compact Riemannian manifold with boundary.
We show that $M$ is Gromov--Hausdorff close to a
convex Euclidean region $D$ of the same dimension
if the boundary distance function of $M$ is $C^1$-close
to that of~$D$.
More generally, we prove the same result under the assumptions
that the boundary distance function of $M$ is $C^0$-close to
that of $D$, the volumes of $M$ and $D$ are almost equal,
and volumes of metric balls in $M$ have a certain lower bound
in terms of radius.
\end{abstract}

\maketitle

\section{Introduction}

Let $M$ be a compact Riemannian manifold with boundary.
For $x,y\in M$, we denote by $d_M(x,y)$ the Riemannian distance
between $x$ and $y$, that is the length of a shortest curve
connecting $x$ and~$y$. The \textit{boundary distance function},
denoted by $bd_M$, is the restriction of $d_M$ to $\pd M\times\pd M$.

In some cases $M$ is uniquely determined by $bd_M$ (up to an isometry
fixing the boundary); such Riemannian manifolds $M$ are called \textit{boundary rigid}.
Michel \cite{michel} conjectured that every simple Riemannian manifold
(that is, such that the boundary is strictly convex and all geodesics are minimizing
and free of conjugate points) is boundary rigid. This conjecture is proved
in dimension~2 by Pestov and Uhlmann \cite{pestov-uhlmann} and in some
partial cases in higher dimensions
(cf.\ \cite{michel}, \cite{gromov-frm}, \cite{BCG}, \cite{CK98}, \cite{BI10}).
In particular, it is shown in \cite{BI10} that,
if $M$ is a region in $\R^n$ with a Riemannian metric which
is sufficiently close (in~$C^2$) to the Euclidean metric $g_e$,
then $M$ is boundary rigid. In other words, if a Riemannian metric $g$ on $D$
defines the same boundary distance function as some almost Euclidean
metric $g'$, then $g$ is isometric to~$g'$.

This raises the following stability question: if the boundary
distance function of a metric $g$ is close to that of $g_e$ in a suitable
topology, is $g$ necessarily close to $g_e$ in $C^r$, $r\ge 2$
(up to an isometry fixing the boundary)?
The answer is known to be affirmative in a local variant of the question,
namely under the assumption that the $C^m$-norm of $g$, for a suitable $m>r$,
is a priori bounded (cf.\ \cite{W} and, for a more general result,~\cite{SU}).
However the global stability question (without further assumptions on $g$) remains open.

In this paper we give an affirmative answer to a weaker variant
of this question, namely we show that $g$ is close to $g_e$ in the
Gromov--Hausdorff topology.
The assumptions on the boundary distance function are also relatively weak:
it should be only $C^1$-close to the boundary distance function of the Euclidean metric.
The precise statement is the following:

\begin{theorem}\label{t-C1}
Let $D\subset\R^n$ be a strictly convex compact region with a smooth boundary.
Then for every $\ep>0$ there exists $\delta>0$ such that the following holds.
Let $M$ be a Riemannian manifold such that $\pd M=\pd D$,
$bd_M$ is $C^1$-smooth on $\pd D\times\pd D\setminus\Delta$
where $\Delta$ is the diagonal of $\pd D\times\pd D$,
and
$$
\|bd_M-bd_D\|_{C^1(\pd D\times\pd D\setminus\Delta)}<\delta .
$$
Then $d_{GH}(M,D)<\ep$ where $d_{GH}$ is the Gromov--Hausdorff distance.
\end{theorem}

Here ``strictly convex'' means that $\pd D$ contains no
straight line segment.
We refer to \cite[\S3A]{Gro99} or \cite[\S7.3]{BBI} for the definition
of the Gromov--Hausdorff distance. For the purposes of this paper,
the following criterion is sufficient \cite[Corollary 7.3.28]{BBI}:
for metric spaces $X$ and $Y$, one has $d_{GH}(X,Y)\le 2\ep$ if there
is a map $f:X\to Y$ such that $f(X)$ is an $\ep$-net in $Y$
(that is, the $\ep$-neighborhood of $f(X)$ covers $Y$) and
$$
|d_Y(f(x),f(x'))-d_X(x,x')|\le\ep
$$
for all $x,x'\in X$. Such maps are referred to as \textit{$\ep$-approximations}.

The boundary distance function $bd_M$ is not differentiable at the diagonal;
this is why the theorem involves
the $C^1$ norm on $\pd D\times\pd D\setminus\Delta$. Alternatively, one may require that
$\|bd_M^2-bd_D^2\|_{C^1}$ is small and the metric tensors of $M$ and $D$
restricted to $\pd D$ are $C^0$-close to each other.

Theorem \ref{t-C1} is proved in section \ref{sec-c1proof}. Here is a sketch
of the proof.
For simplicity, assume that $\pd M$ is strictly convex
(a non-convex boundary requires more technical details,
see section~\ref{sec-c1proof}).
Then the fact that $bd_M$ is $C^1$ implies that all
geodesics in $M$ are minimizing.
For such metrics, Santal\'o's integral geometric formula
(cf.\ section \ref{sec-santalo}) allows one to express the total volume
of $M$ in terms of the boundary distance function and its derivatives.
Applying this formula to $M$ and $D$ yields that $\vol(M)\approx\vol(D)$.
Since all geodesics in $M$ are minimizing, the exponential map at every
point is injective. Then Croke's local isoembolic inequality
(cf.\ section \ref{sec-isoembolic})
yields a uniform lower
bound for volumes of metric balls in $M$ in terms of radii.
With these observations, Theorem~\ref{t-C1} follows from
Theorem \ref{t-C0} (see below) which requires only $C^0$ closeness
of the boundary distance functions but includes volume related assumptions.

For a set $A\subset M$ and $r>0$, we denote by $U_r(A)$
the metric $r$-neighborhood of $A$, that is,
$$
U_r(A)=\{x\in M:\dist_M(x,A)<r\}
$$
where $\dist_M(x,A)=\inf_{y\in A} d_M(x,y)$.
By $B_r(x)$ we denote the metric ball
of radius $r$ centered at~$x\in M$;
that is $B_r(x)=U_r(\{x\})$.

\begin{theorem}\label{t-C0}
Let $D\subset\R^n$ be a convex compact region
and $\lambda>0$ a positive constant.
Then for every $\ep>0$ there exists $\delta=\delta(D,\lambda,\ep)>0$
such that the following holds:
if a compact Riemannian $n$-manifold $M$ 
with $\pd M=\pd D$ satisfies
\begin{enumerate}
 \item $|d_M(x,y) - d_D(x,y)| < \delta$ for all $x,y\in\pd M=\pd D$,
 \item  $\vol(M\setminus U_\delta(\pd M))<\vol(D)+\delta$,
 \item  $\vol(B_r(x))\ge\lambda r^n$ for all $x\in M$
and all $r\ge\delta$ such that $B_r(x)\cap \pd M=\emptyset$,
\end{enumerate}
then $d_{GH}(M,D)<\ep$.
\end{theorem}

\begin{remark}
If $d_M(x,y) \ge d_D(x,y)$ for all $x,y\in\pd M=\pd D$,
then $\vol(M)\ge\vol(D)$ and the equality $\vol(M)=\vol(D)$
implies that $M$ is isometric to $D$. This can be shown as follows.
One may assume that $D$ is contained in the unit cube $I^n$.
Replacing $D\subset I^n$ by $M$ yields a piecewise Riemannian
manifold whose boundary is identified with $\pd I^n$ in such a way
that the distances between opposite faces are no less than~1.
Then Besikovitch inequality \cite{Bes} (see also \cite[\S7.1]{gromov-frm})
implies that the volume of this space is at least~1, and in the
case of equality the space must be isometric to~$I^n$.
Thus Theorem \ref{t-C0} is in a sense a stability estimate
in the equality case of Besikovitch inequality.
\end{remark}

The following example shows that both volume assumptions in Theorem \ref{t-C0} are necessary.

\begin{example}
Let $D\subset\R^n$ be the standard unit ball and $D_r$ be a
ball of radius $r\ll\delta$ with the same center. Remove $D_r$ from $D$
and replace it with either a big round $n$-dimensional sphere with
a similar ball removed,
or a closed-up cylinder $(\pd D_r\times [0,L])\cup (D_r\times\{L\})$
where $L\gg 1$.
Smoothening the resulting piecewise Riemannian metric yields an example
of $M$ such that $\|bd_M-bd_D\|_{C^0}<\delta$ but $d_{GH}(M,D)\ge 1$.
Only the second assumption of Theorem~\ref{t-C0} is violated in the big sphere example,
and only the third one in the cylinder example.
\end{example}

Theorem \ref{t-C0} is proved in sections \ref{sec-coord} and \ref{sec-c0proof}.
In section \ref{sec-coord} we use special distance-like functions on $M$
to construct a Lipschitz map $\phi:M\to\R^n$ which is volume non-increasing and
whose image approximates~$D$.
Then in section \ref{sec-c0proof} we show that
$\phi$ almost preserves distances up to a small additive term, and hence
is an $\ep$-approximation. 
The key points of the proof are the assertion \eqref{e-phiv}
of Proposition \ref{p-phi} and Lemma \ref{l-phi-almost-injective}.

\begin{remark}
The first assumption in Theorem \ref{t-C0}
can be replaced by a weaker one
that does not require identifying $\pd M$ with $\pd D$:
there is a continuous map $F:\pd M\to\pd D$
of nonzero degree mod 2 such that
$
 |d_M(x,y) - d_{\R^n}(F(x),F(y))| < \delta
$
for all $x,y\in\pd M$.
This is what is actually used in the proof, see section~\ref{sec-coord}.
\end{remark}

\noindent\textit{Acknowledgement}.
Many ideas used in this paper arose from joint work with Dima Burago
(\cite{BI10} and \cite{BI-new}).
I am grateful to Anton Petrunin who suggested a simple proof
of Lemma \ref{l-aux1}
and to Sergei Buyalo for his helpful remarks about the text.

\section{Preliminaries}
\label{sec-prelim}

In this section we state some results
used throughout the paper.

\subsection{Area inequality}
Let $M^n$ and $M_1^{n_1}$ be Riemannian manifolds
and $f\colon M\to M_1$ a Lipschitz map.
By Rademacher's Theorem (cf.~\cite[3.1.6]{federer}),
$f$ is differentiable almost everywhere on $M$.
Let $x\in M$ be a point where $f$ is differentiable.
The ($n$-dimensional) Jacobian of $f$ at $x$,
denoted by $Jf(x)$, is the $n$-dimensional volume
of the image of a unit cube in $T_xM$ under
the derivative $d_xf\colon T_xM\to T_{f(x)}M_1$.
We need the following inequality
which is an easy corollary of the area formula
for Lipschitz maps \cite[3.2.3]{federer}.

\begin{prop}
For every measurable set $A\subset M$, one has
$$
\vol_n(f(A)) \le \int_A Jf(x)\,d\vol_n(x)
$$
where $\vol_n$ denotes the $n$-dimensional
Hausdorff measure.
In particular, if $Jf\le 1$ a.e., then $f$ does not increase
$n$-dimensional volumes.
\end{prop}

\subsection{Santal\'o's formula}
\label{sec-santalo}

In order to deduce Theorem \ref{t-C1} from Theorem \ref{t-C0},
we need some integral geometry in the space of geodesics.
Let $M$ be a compact Riemannian manifold with boundary.
We denote by $SM$ the unit tangent bundle of $M$.
For $p\in\pd M$, denote by $\nu(p)$ the unit inner
normal to $\pd M$ and by $S^+_pM$ a hemisphere in
$T_pM$ defined by
$$
S^+_pM=\{v\in S_pM:\langle v,\nu(p)\rangle\ge 0\} .
$$

By a \textit{geodesic} in $M$ we mean a unit-speed
curve in $M$ which is a geodesic of the Riemannian
metric and does not have points on $\pd M$ except possibly endpoints.
For a unit tangent vector $v\in S_pM$, we denote by $\gamma_v$ the
maximal geodesic $\gamma:[0,a]\to M$ or $\gamma:[0,+\infty)\to M$
defined by initial data $\gamma(0)=p$ and $\dot\gamma(0)=v$.
The length of $\gamma_v$ is denoted by $\ell(v)$ or $\ell_M(v)$.

The standard Liouville measure $\mu_L$ on $SM$ is defined by
$$
 \mu_L(A) = \int_M \vol_{S_pM} (A\cap S_pM) \, d\vol_M(p) 
$$
for every measurable set $A\subset SM$,
where $\vol_{S_pM}$ is the standard $(n-1)$-dimensional
volume on the (Euclidean) sphere $S_pM$.
In particular, $\mu_L(SM)=\omega_{n-1} \vol(M)$
where $\omega_{n-1}$ is the volume of the unit sphere in $\R^n$.
The Liouville measure of a set invariant under a geodesic flow
can be recovered from its slice by the boundary
(cf.\ \cite[\S\S19.4--19.5]{Santalo}, \cite{michel}, or \cite[p.~60]{gromov-frm}),
namely the following holds

\begin{prop}
\label{p-santalo}
Let $A\subset \bigcup_{p\in\pd M} S^+_pM$ be a Borel measurable set
and let $\Phi(A)\subset SM$ be the trajectory of $A$ under the
geodesic flow (that is, $\Phi(A)$ is the set of velocity vectors
of all geodesics of the form $\gamma_v$ where $v\in A$). Then
$$
 \mu_L(\Phi(A)) = \int_{\pd M} d\vol_{\pd M}(p)
 \int_{A\cap S^+_pM} \ell(v) \cos\angle (v,\nu(p))\, d\vol_{S_pM}(v) .
$$
\end{prop}

\begin{remark}
If a geodesic $\gamma_v$ (where $v\in S^+_pM$) is a unique minimizing
geodesic between boundary points $p$ and $q$, then the
angle $\angle (v,\nu(p))$ is uniquely determined by
the derivative at $p$ of the function $bd_M(\cdot,q)$,
cf.\ Lemma \ref{l-dbd}.
Thus Proposition \ref{p-santalo} implies that the total volume
of a simple Riemannian manifold $M$
is uniquely determined by $bd_M$.
\end{remark}

\subsection{Local isoembolic inequality}
\label{sec-isoembolic}
M.~Berger \cite{Be80} proved that the volume of a closed
Riemannian manifold $M^n$ is bounded below by
the $n$th power of the injectivity radius
times a constant depending on~$n$
(the equality is attained when $M$ is a round $n$-sphere).
This fact is often referred to as the \textit{isoembolic inequality}.
We need the following ``local'' version of
this inequality, proved by C.~Croke.

\begin{prop}[{\cite[Proposition 14]{Croke80}}]
\label{p-isoembolic}
Let $M^n$ be a complete Riemannian manifold, possibly with boundary.
Let $x\in M$ and $r>0$ be such that $B_r(x)\cap\pd M=\emptyset$
and every geodesic segment contained in $B_r(x)$ is minimizing
(i.e.\ is a shortest path between its endpoints).
Then
\begin{equation}
\label{e-isoembolic}
 \vol(B_r(x)) \ge c r^n
\end{equation}
for some explicit constant $c=c(n)>0$.
\end{prop}

\begin{remark}
In \cite{Croke80}, the result is stated only for boundaryless manifolds,
but the proof uses only the fact that the ball
in question does not reach the boundary.
Indeed, \eqref{e-isoembolic} follows immediately from
the identity $\vol(B_r(x))=\int_0^r \vol_{n-1}(\pd B_t(x))\,dt$
(which holds for any complete Riemannian manifold $M$, $x\in M$ and $r>0$ such that
$B_r(x)\cap\pd M=\emptyset$) and an isoperimetric inequality
$$
\frac{\vol_{n-1}(\pd B_t(x))}{\vol(B_t(x))^{(n-1)/n}}\ge \text{const}(n)
$$
(Theorem 11 of \cite{Croke80}) which holds for any region
(in place of $B_t(x)$) where all geodesics are minimizing.
\end{remark}

\section{Distance-like coordinates}
\label{sec-coord}

This section is the first part of the proof of Theorem \ref{t-C0}.
Here we construct a Lipschitz map $\phi\colon M\to\R^n$
(our would-be Gromov--Hausdorff approximation)
and establish some of its technical properties
(summarized in Proposition \ref{p-phi}).

Let $M$ satisfy the assumptions of Theorem \ref{t-C0} for a small $\delta$.
We fix $D$ and $\lambda$ 
and omit dependence on them in our notations.
We denote by $\ep(\delta)$ various quantities depending
on $\delta$ and tending to~0  as $\delta\to 0$.
In this notation, the assertion of Theorem \ref{t-C0}
is that $d_{GH}(M,D)<\ep(\delta)$.
The notation $A\approx B$ is an abbreviation
for $|A-B|<\ep(\delta)$.
Denote $M'=M\setminus U_\delta(\pd M)$.

To avoid confusion in notation caused by identifying $\pd M$ with $\pd D$,
we replace the first assumption of Theorem \ref{t-C0} by the following:
there is a continuous map $F:\pd M\to\pd D$
of nonzero degree mod 2 such that
\be\label{e-F}
 |d_M(x,y) - d_{\R^n}(F(x),F(y))| < \delta
\qquad\text{for all $x,y\in\pd M$}.
\ee
We fix such a map $F$ for the rest of this section.


For a unit vector $v\in S^{n-1}\subset\R^n$,
define a linear function $L_v:\R^n\to\R$ by
$$
L_v(x)=\langle x,v\rangle
$$
where $\langle \cdot,\cdot\rangle$ is the scalar product in $\R^n$,
and a function $\phi_v:M\to\R$ by
\begin{equation}
\label{phidef}
 \phi_v(x) = \inf_{y\in\pd M} \{d_M(x,y) + L_v(F(y))\} 
\end{equation}
where $F$ is the map from \eqref{e-F}.
Note that this function is 1-Lipschitz on $M$ since it is
a pointwise infimum of 1-Lipschitz functions.
Define a map $\phi:M\to\R^n$ by
$$
 \phi(x) = (\phi_{e_1}(x),\dots,\phi_{e_n}(x))
$$
where $(e_1,\dots,e_n)$ is the standard basis of $\R^n$.
Obviously $\phi$ is $n$-Lipschitz.
Since the coordinate functions of $\phi$ are 1-Lipschitz,
its Jacobian at any point of differentiability
is no greater than~1.
Therefore $\phi$ is volume non-increasing.

Our ultimate goal is to show that $\phi$ is an $\ep(\delta)$-approximation
of a small neighborhood of $D$ in $\R^n$.
In this section we prove the following proposition.

\begin{prop}\label{p-phi}
For every unit vector $v\in\R^n$ the following holds.
\begin{enumerate}
 \item[\refstepcounter{equation}\label{e-phi-maxlip}\eqref{e-phi-maxlip}]
For every $x\in M$ there is a point $y\in\pd M$ such that
$\phi_v(x) = \phi_v(y) + d_M(x,y)$.
 \item[\refstepcounter{equation}\label{e-phibd}\eqref{e-phibd}]
$|\phi(x)-\phi(y)|\approx d_M(x,y)$ for all $x,y\in\pd M$.
 \item[\refstepcounter{equation}\label{e-phi-vol}\eqref{e-phi-vol}]
$\vol(\phi(E))\ge\vol(E)-\ep(\delta)$
for every measurable $E\subset M'$.
 \item[\refstepcounter{equation}\label{e-phi-image}\eqref{e-phi-image}]
$\phi(M)$ and $\phi(\pd M)$ are within Hausdorff distance $\ep(\delta)$
from $D$ and $\pd D$, resp.
 \item[\refstepcounter{equation}\label{e-phiv}\eqref{e-phiv}]
$\phi_v(x)\approx L_v(\phi(x))$ for all $x\in M$.
\end{enumerate}
\end{prop}

The proof of Proposition \ref{p-phi} occupies the rest of this section.
Most of the assertions are nearly trivial; only \eqref{e-phiv}
requires some work.

\medskip\noindent\textbf{Proof of (\ref{e-phi-maxlip}).}
Fix $x\in M$ and let $y\in\pd M$ be a point where
the infimum in \eqref{phidef} is attained.
Then $\phi_v(x)=d_M(x,y)+L_v(F(y))$.
Since $\phi_v$ is 1-Lipschitz, it follows that
$\phi_v(y)\ge L_v(F(y))$.
On the other hand,
$$
 \phi_v(y) \le d_M(y,y)+L_v(F(y))=L_v(F(y))
$$
by the definition of $\phi_v$.
Thus $\phi_v(y)=L_v(F(y))$ and \eqref{e-phi-maxlip} follows.
\qed

\medskip\noindent\textbf{Proof of (\ref{e-phibd}).}
For every $x\in\pd M$ and every unit vector $v\in\R^n$ we have
\begin{equation}
\label{phi1}
  |\phi_v(x)- L_v(F(x))| \le \delta
\end{equation}
where $F$ is the map from \eqref{e-F}. Indeed, for every $y\in\pd M$,
$$
 d_M(x,y) + L_v(F(y)) \ge d_{\R^n}(F(x),F(y))+L_v(F(y))-\delta \ge L_v(F(x))-\delta
$$
since $L_v$ is 1-Lipschitz. Hence $\phi_v(x)\ge L_v(F(x))-\delta$.
On the other hand, substituting $y=x$ under the infimum in \eqref{phidef}
yields that $\phi_v(x)\le L_v(F(x))$, and \eqref{phi1} follows.
Since $(L _{e_1},\dots,L_{e_n})=\id_{\R^n}$,
\eqref{phi1} implies
\be\label{phi-dm}
 |\phi(x)-F(x)| \le n\delta\qquad\text{for all $x\in\pd M$}.
\ee
This and \eqref{e-F} imply that
$$
 |\phi(x)-\phi(y)|\approx |F(x)-F(y)| \approx d_M(x,y)
$$
for all $x,y\in\pd M$.
\qed

\medskip\noindent\textbf{Proof of (\ref{e-phi-vol}).}
By \eqref{phi-dm},
$\phi(\pd M)\subset U_{n\delta}(F(\pd M))=U_{n\delta}(\pd D)$ and
moreover $\phi|_{\pd M}$ is homotopic to $F$ in $U_{n\delta}(\pd D)$.
Therefore $\phi$ has degree 1 over any point of
$D\setminus U_{n\delta}(\pd D)$, hence
$\phi(M)\supset D\setminus U_{n\delta}(\pd D)$.
Furthermore,
\be\label{phi-near-dm}
 \phi(M\setminus M')= \phi(U_\delta(\pd M))\subset U_{2n\delta}(\pd D),
\ee
since $\phi$ is $n$-Lipschitz. Hence
\be\label{nbhd-neg}
 \phi(M')\supset D\setminus U_{2n\delta}(\pd D).
\ee
Since $\phi$ is volume non-increasing, we have
$$
 \vol(M') \ge \vol(\phi(M'))
 \ge \vol(D\setminus U_{2n\delta}(\pd D))>\vol(D)-\ep(\delta)
 > \vol(M')-\ep(\delta)
$$
by \eqref{nbhd-neg} and the second assumption of Theorem \ref{t-C0}.
Let $E\subset M'$ be a measurable set. Then
$$
 \vol(\phi(E))+\vol(\phi(M'\setminus E)) \ge \vol(\phi(M'))>\vol(M') - \ep(\delta) .
$$
On the other hand, $\vol(\phi(M'\setminus E)) \le \vol(M'\setminus E)$
since $\phi$ is volume non-increasing.
Hence $\vol(\phi(E))> \vol(M')-\vol(M'\setminus E)-\ep(\delta)=\vol(E)-\ep(\delta)$.
\qed

\medskip\noindent\textbf{Proof of (\ref{e-phi-image}).}
The assertion about $\phi(\pd M)$ follows from \eqref{phi-dm}
and the fact that $F(\pd M)=\pd D$.
By \eqref{nbhd-neg},
$D$ is contained in a small neighborhood of $\phi(M)$.
It remains to show that $\phi(M)$ is contained in a small neighborhood of~$D$.

Let $p\in M$ and $r=\dist_{\R^n}(\phi(p),D)$.
We are to prove that $r<\ep(\delta)$.
Suppose that $r>4n\delta$ and consider a metric ball
$B=B_{r/2n}(p)$.
Since $\phi$ is $n$-Lipschitz, we have
$$
 \phi(B) \subset B_{r/2}(\phi(p))
 \subset \R^n\setminus U_{2n\delta}(D),
$$
hence $B\subset M'$ by \eqref{phi-near-dm}.
Hence by \eqref{e-phi-vol} and the third assumption of Theorem \ref{t-C0}
we have
$$
 \vol(\phi(B)) > \vol(B)-\ep(\delta) \ge \lambda(r/2n)^n - \ep(\delta) .
$$
This and \eqref{nbhd-neg} imply that
$$
 \vol(M') \ge \vol(\phi(M'))\ge \vol(D\setminus U_{2n\delta}(\pd U))+\vol(\phi(B))
 > \vol(D) +  \lambda(r/2n)^n - \ep(\delta) .
$$
On the other hand, $\vol(M')<\vol(D)+\ep(\delta)$ by the second requirement
of Theorem \ref{t-C0}.
Therefore $\lambda(r/2n)^n < \ep(\delta)$,
hence $r<\ep(\delta)$.
\qed

\medskip\noindent\textbf{Proof of (\ref{e-phiv}).}
We need one more construction and some lemmas.

Fix a unit vector $v\in\R^n$ and
an orthonormal basis $(v_1,v_2,\dots,v_n)$ in $\R^n$ such that $v_1=v$.
Define a linear map $I:\R^n\to\R^{2n}$ by
$$
 I = \tfrac1{\sqrt2}(L_{e_1},\dots,L_{e_n},L_{v_1},\dots,L_{v_n})
$$
and a Lipschitz map $\Phi:M\to\R^{2n}$ by
$$
 \Phi = \tfrac1{\sqrt2}(\phi_{e_1},\dots,\phi_{e_n},\phi_{v_1},\dots,\phi_{v_n}) .
$$
Observe that $I$ is a linear isometric embedding,
$\Phi$ is a $2n$-Lipschitz map
and
\be\label{Phi-dm}
|\Phi(x)-I(F(x))|\le 2n\delta
\ee
for all $x\in\pd M$ (by \eqref{phi1} and \eqref{phi-dm}).
Therefore
\be\label{Phi-near-dm}
 \Phi(M\setminus M') = \Phi(U_\delta(\pd M)) \subset U_{4n\delta}(I(\pd M)).
\ee

\begin{lemma}\label{l-Phi-vol}
$\Phi$ does not increase $n$-dimensional volumes.
\end{lemma}

\begin{proof}
The statement follows from the fact that the $2n$ coordinate
functions of $\Phi$ are $(1/\sqrt2)$-Lipschitz.

Indeed, let $p$ be a point
of differentiability of $\Phi$ and let $Q$ be the pull-back of the Euclidean
structure of $\R^{2n}$ by $d_p\Phi$. That is, $Q$ is a quadratic form on $T_pM$
defined by $Q(w)=|d_p\Phi(w)|^2$ for all $w\in T_pM$. Then
$$
 \trace Q = \sum_{i=1}^{2n} \trace (d_p\Phi_i)^2 = \sum_{i=1}^{2n} \|d_p\Phi_i\|^2
 \le \sum_{i=1}^{2n} \frac12=n,
$$
since $\|d_p\Phi_i\|\le 1/\sqrt2$ for all $i=1,\dots,2n$.
Hence
$$
 \det Q \le \left(\tfrac1n\trace Q\right)^n \le 1.
$$
Here $\Phi_i$, $i=1,\dots,n$, are the coordinate functions of $\Phi$
(that is, $\Phi_i=\phi_{e_i}$ or $\Phi_i=\phi_{v_{i-n}}$)
and all traces and determinants are with respect to the Euclidean
structure on $T_pM$ defined by the Riemannian metric. 
The last inequality means that the $n$-dimensional Jacobian of $\Phi$ at $p$
is no greater than~1, hence $\Phi$ does not increase $n$-dimensional volumes.
\end{proof}

Our next goal is to show that $\Phi(M)$ is contained
in a small neighborhood of the subspace $I(\R^n)$ in $\R^{2n}$
(cf.\ Lemma \ref{Phi-nbhd}).
The following lemma is an intermediate step towards this.

\begin{lemma}\label{l-aux1}
For every fixed $r>0$, one has
$
 \vol (\Phi^{-1}(\R^{2n}\setminus U_r(I(\R^n)))) < \ep(\delta) 
$.
\end{lemma}

\begin{proof}
Denote $W=I(\R^n)$. There is a 1-Lipschitz map $P:\R^{2n}\to W$
and a constant $c>0$ such that
$P|_{I(D)}=\id_{I(D)}$ and
\be\label{shrink}
 J_n P(x)\le 1-c \qquad\text{for all $x\in\R^{2n}\setminus U_r(W)$,}
\ee
where $J_n$ denotes the $n$-dimensional Jacobian.
Indeed, let $Q\subset\R^{2n}$ be a solid ellipsoid
such that $I(D)\subset Q\subset U_{r/2}(W)$
and let $P_0:\R^{2n}\to Q$ be the nearest-point projection to $Q$.
Then $P_0$ is 1-Lipschitz and satisfies \eqref{shrink} for some $c>0$.
A desired map $P$ can be obtained by composing $P_0$
with the orthogonal projection to~$W$.

Define a map $f:M\to\R^n$ by $f=I^{-1}\circ P\circ\Phi$.
Note that $f$ is volume non-increasing since so are
$\Phi$, $P$ and $I^{-1}$.
Let $E=\Phi^{-1}(\R^{2n}\setminus U_r(W))$, then
\eqref{shrink} implies that
$$
 \vol(f(E)) \le (1-c)\vol(E) .
$$
By \eqref{Phi-near-dm} we have
$\Phi(M\setminus M')\subset U_{4n\delta}(W)$,
hence $E\subset M'$ provided that $\delta<r/4n$.
By \eqref{Phi-dm}, we have $f\approx F$ on $\pd M$.
Similarly to the proof of \eqref{e-phi-vol},
this and the fact that $f$ is volume non-increasing
imply that
$$
 \vol(f(E)) > \vol(E)-\ep(\delta) .
$$
Now the two above inequalities on $\vol(f(E))$
imply that $\vol(E)<\ep(\delta)/c=\ep(\delta)$.
\end{proof}

\begin{lemma}\label{Phi-nbhd}
$\Phi(M)\subset U_{\ep(\delta)}(I(\R^n))$.
\end{lemma}

\begin{proof}
Suppose the contrary. Then there exists $r>0$ such that
for every $\delta>0$ there is a manifold $M$ satisfying
the assumptions of Theorem \ref{t-C0} and
maps $\Phi$ and $I$ constructed as above such that
$\dist(\Phi(p),I(\R^n))\ge r$ for some $p\in M$.

Choose such $M$, $\Phi$, $I$ and $p$ for a sufficiently small $\delta$.
We may assume that $r>8n\delta$. Consider a metric ball
$B=B_{r/4n}(p)$. Since $\Phi$ is $2n$-Lipschitz,
we have
$$
 \Phi(B) \subset B_{r/2}(\Phi(p)) \subset \R^{2n}\setminus U_{r/2}(I(\R^n)) .
$$
Therefore $\vol(B) < \ep(\delta)$ by Lemma \ref{l-aux1},
and the 3rd assumption of Theorem \ref{t-C0} implies that
$r<\ep(\delta)$, a contradiction.
\end{proof}

Now let $P_i:\R^{2n}\to\R$, $i=1,\dots,2n$, denote the coordinate projections
multiplied by $\sqrt2$.
Observe that $\phi_v=P_{n+1}\circ\Phi$ and $L_v=P_{n+1}\circ I$.
Define $P:\R^{2n}\to\R^n$ by $P=(P_1,\dots,P_n)$,
then $P\circ I=\id_{\R^n}$ and $P\circ\Phi=\phi$.
By Lemma \ref{Phi-nbhd}, for a given $x\in M$
there is a point $x'\in\R^n$
such that $|\Phi(x)-I(x')|<\ep(\delta)$.
Then
$$
 |\phi(x)-x'| = |P(\Phi(x))-P(I(x'))| \le \sqrt2 |\Phi(x)-I(x')| < \ep(\delta)
$$
where the first inequality follows from the fact that $P$ is $\sqrt2$-Lipschitz.
Hence
$$
  |L_v(\phi(x))-L_v(x')| \le |\phi(x)-x'|< \ep(\delta) .
$$
Furthermore,
$$
 |\phi_v(x)-L_v(x')| = |P_{n+1}(\Phi(x))-P_{n+1}(I(x'))| \le \sqrt2|\Phi(x)-I(x')|< \ep(\delta)
$$
The last two inequalities yield \eqref{e-phiv}.
This completes the proof of Proposition \ref{p-phi}.

\section{Estimating distances in $M$}\label{sec-c0proof}

In this section we finish the proof of Theorem~\ref{t-C0}
by showing that $\phi$ almost preserves the distances
(up to an additive term $\ep(\delta)$).

\begin{lemma}\label{l-phi-almost-short}
$|\phi(x)-\phi(y)|< d_M(x,y)+\ep(\delta)$
for all $x,y\in M$.
\end{lemma}

\begin{proof}
Let $v$ be a unit vector in $\R^n$ such that $\phi(x)-\phi(y)$
is a nonnegative multiple of $v$. Then
$$
 |\phi(x)-\phi(y)| = L_v(\phi(x)) - L_v(\phi(y)) \approx \phi_v(x)-\phi_v(y) \le d_M(x,y) .
$$
Here the first relatin follows from the definition of $L_v$,
the second from \eqref{e-phiv}, and the third from the fact that $\phi_v$ is 1-Lipschitz.
\end{proof}

\begin{lemma}\label{l-dist-dm}
$ \dist_M(x,\pd M)\approx\dist_{\R^n}(\phi(x),\pd D) $
for all $x\in M$.
\end{lemma}

\begin{proof}
Fix $x\in M$. Lemma \ref{l-phi-almost-short} implies that
$$
 \dist_M(x,\pd M) > \dist_{\R^n}(\phi(x),\phi(\pd M)) - \ep(\delta)
 > \dist_{\R^n}(\phi(x),\pd D)) - \ep(\delta)
$$
since $\phi(\pd M)$ is contained in a small neighborhood
of $\pd D$ (cf.~\eqref{e-phi-image}).

To prove the opposite inequality,
let $p\in\pd D$ be a point of $\pd D$ nearest to $\phi(x)$
and $v$ the inner normal to $\pd D$ at $p$
(or, if $\pd D$ has no tangent hyperplane at $p$, a normal
to any supporting hyperplane).
If $\phi(x)\in D$, then $\phi(x)-p$ is a nonnegative multiple of $v$
and therefore
$$
  L_v(\phi(x)) = L_v(p) + |p-\phi(x)| = L_v(p)+\dist_{\R^n}(\phi(x),\pd D)
$$
by the definition of $L_v$.
If $\phi(x)\notin D$, then $\dist_{\R^n}(\phi(x),\pd D)\approx 0$
by \eqref{e-phi-image}, hence $\phi(x)\approx p$ and
$L_v(\phi(x)) \approx L_v(p)$.
In both cases we have
\be\label{bdist1}
 \phi_v(x)\approx L_v(\phi(x))\approx  L_v(p)+\dist_{\R^n}(\phi(x),\pd D)
\ee
where the first relation follows from \eqref{e-phiv}.
By \eqref{e-phi-maxlip} and \eqref{e-phiv},
$$
 \phi_v(x) = \phi_v(y)+d_M(x,y)\approx L_v(\phi(y)) + d_M(x,y)
$$
for some $y\in\pd M$.
Since $D$ is convex, $p$ is a point of minimum of $L_v|_{\pd D}$.
Since $\phi(y)$ is close to $\pd D$
(by~\eqref{e-phi-image}),
it follows that $L_v(\phi(y)) > L_v(p)-\ep(\delta)$.
Thus
$$
 \phi_v(x) > L_v(p) + d_M(x,y) -\ep(\delta) \ge L_v(p) + \dist_M(x,\pd M)-\ep(\delta) .
$$
This and \eqref{bdist1} imply that
$
 \dist_M(x,\pd M) < \dist_{\R^n}(\phi(x),\pd D) + \ep(\delta)
$.
\end{proof}

\begin{lemma}\label{l-curve}
For every $r>0$ there is a $\delta_0>0$ such that the following holds:
if $\delta<\delta_0$, $x,y\in M$,
$|\phi(x)-\phi(y)| \le r$ and
$\dist_{\R^n}(\phi(x),\pd D) \ge 3r$,
then there is a curve $\gamma$ connecting $x$ and $y$ in $M$
such that $\phi(\gamma)\subset B_{2r}(\phi(x))$.
\end{lemma}

\begin{proof}
We may assume that $2n\delta<r$. Let $B=B_{2r}(\phi(x))$. 
Since $\phi(M\setminus M')\subset U_{2n\delta}(\pd D)$
(cf.~\eqref{phi-near-dm}), we have
$B\cap \phi(M\setminus M')=\emptyset$,
hence the set $U:=\phi^{-1}(B)$ is contained in~$M'$.
Let $U_x$ and $U_y$ be the connected components of $U$ containing
$x$ and $y$, respectively. If $U_x=U_y$ then any curve $\gamma$ connecting
$x$ and $y$ in $U$ satisfies the desired condition.

Suppose that $U_x\ne U_y$. Since $\phi$ is $n$-Lipschitz,
we have $B_{r/n}(x)\subset U_x$.
Hence, by the 3rd assumption of Theorem \ref{t-C0},
$\vol(U_x) \ge cr^n$ where $c=\lambda n^{-n}$.
Consider $\phi_x=\phi|_{U_x}$ regarded as a map from $U_x$ to $B$.
This map is proper and hence has a well-defined degree mod~2.
If $\deg_2(\phi_x)=0$, then every regular value of
any smooth approximation of $\phi_x$
has zero or at least two pre-images.
Since $\phi_x$ is volume non-increasing, it follows that
$$
 \vol (\phi(U_x)) \le \tfrac12 \vol(U_x) \le \vol(U_x) - cr^n/2,
$$
contrary to \eqref{e-phi-vol}.
Thus $\deg_2(\phi_x)=1$.

The same argument applies to $U_y$ and a
map $\phi_y=\phi|_{U_y}\colon U_y\to B$,
therefore $\deg_2(\phi_y)=1$ as well.
Hence both $\phi_x$ and $\phi_y$ are surjective, hence
$$
 \vol(\phi(U))=\vol(\phi(U_y)) \le \vol(U_y) \le \vol(U)-\vol(U_x) \le \vol(U)-cr^n,
$$
contrary to \eqref{e-phi-vol}.
\end{proof}

\begin{lemma}\label{l-phi-almost-injective}
For every $r>0$ there exist $\rho>0$ and $\delta_0>0$ such that the following holds.
If $\delta<\delta_0$ and $x,y\in M$ are such that $|\phi(x)-\phi(y)|<\rho$,
then $d_M(x,y)<r$.

In other words, $d_M(x,y)\to 0$ as $|\phi(x)-\phi(y)|\to 0$ and $\delta\to 0$.
\end{lemma}

\begin{proof}
Suppose that $r,\rho>0$ and $x,y\in M$ are such that $d_M(x,y)\ge r$
and $|\phi(x)-\phi(y)|<\rho$.
We are going to obtain a contradiction assuming that
$\rho=cr$ for a suitable constant $c>0$ and $\delta\ll r$.

First consider the case when $\dist_{\R^n}(\phi(x),\pd D)\ge 3\rho$.
By Lemma \ref{l-curve} there is a curve $\gamma$ connecting
$x$ and $y$ such that $\phi(\gamma)\subset B_{2\rho}(\phi(x))$.
Let $N$ be a positive integer such that $N\ge\frac{nr}\rho>N-1$.
Then there are points $x_1,\dots,x_N$ on $\gamma$ such that
$d_M(x,x_k)=(k-1)\rho/n$ for all $k$.
The triangle inequality implies that
the balls $B_k:=B_{\rho/2n}(x_k)$ are disjoint.
Denote $U=\bigcup B_k$.
We may assume that $\rho/2n>2\delta$, then by the third assumption of
Theorem \ref{t-C0} we have
$$
 \vol(U) = \sum \vol(B_k) \ge N\cdot \lambda(\rho/2n)^n = \mu N\rho^n
$$
where $\mu=\lambda(2n)^{-n}$.
Since $\phi$ is $n$-Lipschitz, we have $\phi(B_k)\subset B_{\rho/2}(\phi(x_k))$.
Since $\phi(x_k)\in\phi(\gamma)\subset B_{2\rho}(\phi(x))$ for all~$k$, it follows that
$\phi(U)\subset B_{5\rho/2}(\phi(x))$, hence
$
 \vol(\phi(U)) \le C\rho^n
$
where $C$ is the volume of a Euclidean $n$-ball of radius~$5/2$.
Furthermore, $\phi(U)$ is separated away from $\pd M$ by
distance $\rho/2>2n\delta$, hence $U\subset M'$
(cf.~\eqref{phi-near-dm}) and therefore
$
\vol(\phi(U))>\vol(U)-\ep(\delta)
$
by \eqref{e-phi-vol}.
Thus
$$
 C\rho^n \ge \vol(\phi(U))>\vol(U)-\ep(\delta)
\ge \mu N\rho^n-\ep(\delta)  \ge \mu r\rho^{n-1}-\ep(\delta)
$$
since $N\ge r/\rho$. Fix $\rho=\mu r/2C$ and assume that $\delta$
is so small that the above $\ep(\delta)$ satisfies
$\ep(\delta)<\frac12\mu r\rho^{n-1}$.
Then $C\rho^n >\frac12\mu r\rho^{n-1}$,
hence $r< 2C\rho/\mu=r$, a contradiction.

It remains to consider the case when $\dist_{\R^n}(\phi(x),\pd D)< 3\rho$.
Let $x'$ and $y'$ be points of $\pd M$ nearest to $x$ and~$y$ respectively.
Then Lemma \ref{l-dist-dm} imples that
$$
 d_M(x,x') = \dist_M(x,\pd M) \approx \dist_{\R^n}(\phi(x),\pd D)< 3\rho
$$
and
$$
 d_M(y,y') = \dist_M(y,\pd M) \approx \dist_{\R^n}(\phi(y),\pd D)
 \le \dist_{\R^n}(\phi(x),\pd D) + |\phi(x)-\phi(y)| < 4\rho .
$$
Since $\phi$ is $n$-Lipschitz, it follows that
$$
|\phi(x)-\phi(x')|<3n\rho+\ep(\delta) \quad\text{and}\quad
|\phi(y)-\phi(y')|<4n\rho+\ep(\delta) ,
$$
hence
$$
 |\phi(x')-\phi(y')|<|\phi(x)-\phi(y)|+7n\rho+\ep(\delta)\le (7n+1)\rho+\ep(\delta).
$$
By \eqref{e-phibd},
$$
 d_M(x',y') \approx |\phi(x')-\phi(y')|<(7n+1)\rho+\ep(\delta) .
$$
Therefore
$$
 r\le d_M(x,y) \le d_M(x,x')+d_M(x',y')+d_M(y,y') < (7n+8)\rho+\ep(\delta) .
$$
This is impossible if $\rho\le\frac12(7n+8)^{-1}r$ and $\delta\ll r$.
\end{proof}

\begin{lemma}\label{l-phi-almost-isometry}
$d_M(x,y)\approx |\phi(x)-\phi(y)|$ for all $x,y\in M$.
\end{lemma}

\begin{proof}
By Lemma \ref{l-phi-almost-short}, it suffices to show that
$$
 d_M(x,y) < |\phi(x)-\phi(y)| + \ep(\delta) .
$$
Without loss of generatilty assume that
$\dist_{\R^n}(\phi(x),\pd D) \ge \dist_{\R^n}(\phi(y),\pd D)$.
Let $v$ be a unit vector in $\R^n$ such that $\phi(x)-\phi(y)$
is a nonnegative multiple of~$v$.
By \eqref{e-phi-maxlip} there is a point $z\in\pd M$ such that
$d_M(x,z)=\phi_v(x)-\phi_v(z)$.
Let $\gamma$ be a shortest path from $x$ to~$z$ in~$M$
and $q$ an arbitrary point on~$\gamma$, then $d_M(x,z)=d_M(x,q)+d_M(q,z)$.
Since $\phi_v$ is 1-Lipschitz, we have
$\phi_v(x)-\phi_v(q)\le d_M(x,q)$
and $\phi_v(q)-\phi_v(z)\le d_M(q,z)$.
If any of these two inequalities is strict, adding them yields
that $\phi_v(x)-\phi_v(z)<d_M(x,z)$, contrary to the choice of~$z$.
Thus $\phi_v(x)-\phi_v(q)=d_M(x,q)$.

Lemma \ref{l-phi-almost-short} implies that $d_M(x,q)>|\phi(x)-\phi(q)|-\ep(\delta)$.
By  \eqref{e-phiv} we have
\be\label{q-estimate}
d_M(x,q)=\phi_v(x)-\phi_v(q)\approx L_v(\phi(x))-L_v(\phi(q)) = \langle \phi(x)-\phi(q), v\rangle ,
\ee
therefore
$
 \langle \phi(x)-\phi(q), v\rangle > |\phi(x)-\phi(q)| - \ep(\delta) 
$.
This implies that the vector $\phi(x)-\phi(q)$ is $\ep(\delta)$-close
to a positive multiple of~$v$. Since $q$ is an arbitrary point on $\gamma$,
this means that $\phi(\gamma)$ is contained in an $\ep(\delta)$-neighborhood
of the ray $R:=\{\phi(x)-tv:t\ge 0\}$.

Since $z\in\pd M$, $\phi(z)$ is close to $\pd D$ (cf.~\eqref{e-phi-image}).
Since the curve $\phi(\gamma)$ connects $\phi(x)$ to $\phi(z)\in U_{\ep(\delta)}(\pd D)$,
$\phi(\gamma)\subset U_{\ep(\delta)}(R)$,
$\phi(y)\in R\cap \phi(M)\subset R\cap U_{\ep(\delta)}(D)$
and $D$ is convex,
there are two possibilities:
either $\phi(x)$ is close to $\pd D$ or $\phi(\gamma)$ passes near $\phi(y)$.
In the former case $\phi(y)$ is close to $\pd D$ as well (by our initial
assumption), and the desired assertion follows from Lemma \ref{l-dist-dm}
and \eqref{e-phibd}. In the latter case
consider a point $q\in\gamma$ such that $|\phi(q)-\phi(y)|<\ep(\delta)$.
Since $q\in\gamma$, we have $d_M(x,q)< |\phi(x)-\phi(q)|+\ep(\delta)$
by \eqref{q-estimate}
By Lemma \ref{l-phi-almost-injective},
the inequality $|\phi(q)-\phi(y)|<\ep(\delta)$
implies that $d_M(q,y)<\ep(\delta)$,
therefore
$$
d_M(x,y)\approx d_M(x,q) < |\phi(x)-\phi(q)|+\ep(\delta) \approx |\phi(x)-\phi(y)|
$$
and the lemma follows.
\end{proof}

\begin{proof}[Proof of Theorem \ref{t-C0}]
By Lemma \ref{l-phi-almost-isometry},
$\phi$ is an $\ep(\delta)$-approximation of $\phi(M)\subset\R^n$,
hence $d_{GH}(M,\phi(M))<\ep(\delta)$. By \eqref{e-phi-image},
the Hausdorff distance in $\R^n$ between $\phi(M)$ and $D$ is small,
hence $d_{GH}(\phi(M),D)<\ep(\delta)$. Therefore
$$
 d_{GH}(M,D) \le d_{GH}(M,\phi(M))+d_{GH}(\phi(M),D) < \ep(\delta)
$$
and the theorem follows.
\end{proof}

\section{Proof of Theorem \ref{t-C1}}
\label{sec-c1proof}

Let $M$ be a compact Riemannian $n$-manifold with boundary.
We use the notation introduced in section \ref{sec-santalo},
namely $SM$ denotes the unit tangent bundle of $M$,
$S^+_pM$ (where $p\in\pd M)$ is the hemisphere
of inward-pointing vectors from $S_pM$,
$\gamma_v$ is the maximal forward geodesic
with initial velocity vector $v\in SM$
and $\ell(v)$ or $\ell_M(v)$ is the length of $\gamma_v$.
Clearly $\ell$ is a lower semi-continuous function
from $SM$ to $[0,+\infty]$.

We say that a unit-speed curve $\gamma:[a,b]\to M$ is \textit{minimizing}
(or a \textit{minimizer}, or a \textit{shortest path}) if it realizes the
distance between $\gamma(a)$ and $\gamma(b)$.
Since $M$ is compact, every pair of points is connected by a minimizer.
Note that a minimizer is not necessarily
a geodesic since it may bend along the boundary.

We need some basic facts about minimizers in Riemannian
manifolds with boundary (see e.g.~\cite{ABB}): every
minimizer is $C^1$ and (point-wise) convergence of minimizers
implies  convergence of their tangents.

If two points $x,y\in M$ are such that all shortest paths from  $x$ to $y$
have the same velocity vector at $x$, we denote this vector
by $\overrightarrow{xy}$ and say
that $\overrightarrow{xy}$ is uniquely defined.

\begin{lemma}\label{l-dbd}
If $bd_M$ is differentiable at a point $(x,y)\in\pd M\times\pd M$,
then $\overrightarrow{xy}$ is uniquely defined and the projection
of $\overrightarrow{xy}$ to $T_x\pd M$ equals
the Riemannian gradient of the function $-bd_M(\cdot,y)$.
\end{lemma}

\begin{proof}
This is standard. Denote $f=bd_M(\cdot,y)$
and let $\gamma$ be a shortest path from $x$ to~$y$.
Then the first variation formula
implies that for every $v\in T_x\pd M$ one has
$
 d_xf(v) \le -\langle v, \dot\gamma(0) \rangle
$.
Applying this to $v$ and $-v$ yields that
$
 d_xf(v) = -\langle v, \dot\gamma(0) \rangle
$
for all $v\in T_x\pd M$.
Hence the gradient of $-f$ at $x$ is the
projection of $\dot\gamma(0)$ to $T_x\pd M$.
Since $\dot\gamma(0)\in S_x^+M$, this vector
is uniquely determined by its projection to $T_x\pd M$.
\end{proof}

\begin{lemma}\label{l-bg-minimizing}
If $bd_M$ is differentiable away from the diagonal,
then every geodesic starting at the boundary is minimizing.
In particular, all such geodesics have length bounded above
by $\diam(M)$.
\end{lemma}

\begin{proof}
Let $\gamma:[0,a]\to M$ be a geodesic with $\gamma(0)=p\in\pd M$.
First consider the case when the initial vector
$v:=\dot\gamma(0)$ is not tangent to $\pd M$.

Define a map $f:\pd M\setminus\{p\}\to S_p^+M$
by $f(x)=\overrightarrow{px}$. By the previous lemma,
this map is well-defined and hence continuous.
It is easy to see that
\begin{equation}
\label{e1}
\left|f(x)-u(\exp_{p,\pd M}^{-1}(x))\right|\to 0
\qquad\text{as $x\to p$},
\end{equation}
where $u:T_pM\setminus\{0\}\to S_pM$ is the normalization function
defined by $u(w)=w/|w|$, and $\exp_{p,\pd M}$ is the Riemannian exponential
map of $\pd M$ at~$p$ (restricted to a neighborhood of the origin where
it is injective).
Denote $\alpha=\angle(v,\pd M)$ and let $B$ be a small geodesic ball in $\pd M$
centered at $p$ such that the left-hand side of \eqref{e1} is less than $\alpha$
for all $x\in B$. Then $f|_{\pd B}$ is homotopic to $u\circ\exp_{p,\pd M}^{-1}|_{\pd B}$
as a map from $\pd B$ to $S^+_pM\setminus\{v\}$.
Since $u\circ\exp_{p,\pd M}^{-1}|_{\pd B}$ is a diffeomorphism
from $\pd B$ to the boundary of $S^+_pM$, it follows that
$f$ has degree~1 over~$v$. In particular, $f^{-1}(v)$ is nonempty.
Therefore there is a point $q\in\pd M$ such that $v=\overrightarrow{pq}$.
Then $\gamma$ is an interval of a shortest path from $p$ to $q$
and hence a minimizer.

Now consider the case when $v$ is tangent to the boundary.
Choose a sequence $\{v_i\}$ in the interior of $S^+_pM$
such that $v_i\to v$.
As shown above, the geodesics $\gamma_{v_i}$ are minimizing.
A limit of a subsequence of $\{\gamma_{v_i}\}$ is a minimizer
with endpoints at the boundary and with initial velocity~$v$.
Hence $\gamma$ is an interval of this limit, therefore
it is minimizing.
\end{proof}

Now assume that $M$ satisfies the assumptions of Theorem \ref{t-C0}
for a small $\delta>0$. We continue using the notations
$\ep(\delta)$ and $\approx$ defined in section \ref{sec-coord}.

First observe that the induced Riemannian metric on $\pd M$
at a point $p\in\pd M$ can be recovered from the first derivatives
of a function $bd_M(p,\cdot)$ near $p$. Indeed, for every tangent vector
$v\in T_p\pd M$ and a smooth curve
$\gamma:[0,1]\to\pd M$ with $\dot\gamma(0)=v$ one has
$$
 |v|=\lim_{t\to 0}d_{\gamma(t)} bd_M(p,\cdot) (\dot\gamma(t))
$$
where $|\cdot|$ is the norm defined by the Riemannian metric
and $d_{\gamma(t)}$ denotes the derivative at $\gamma(t)$.
This formula depends continuously on the derivatives of $bd_M$, hence
$$
 \bigl\|{g_M}_{|_{T\pd D}}-{g_{\R^n}}_{|_{T\pd D}}\bigr\|_{C^0} \le \ep(\delta)
$$
where $g_M$ denotes the metric tensor of $M$.

\begin{lemma}\label{l-no-infinite-geodesics}
Every non-minimizing geodesic stays within
distance $\ep(\delta)$ from $\pd M$.
\end{lemma}

\begin{proof}
By Lemma \ref{l-bg-minimizing}, a non-minimizing geodesic
never hits the boundary and therefore can be extended
to infinite length.
Let $\gamma$ be a geodesic parametrized by $[0,+\infty)$ and
$p=\gamma(0)$; we are to prove that $\dist_M(p,\pd M)<\ep(\delta)$.
Consider the set $Z$ of all vectors $v\in S_pM$ such that
the geodesic $\gamma_v$ eventually hits $\pd M$.
By Lemma \ref{l-bg-minimizing}, lengths of these geodesics
are bounded above by $\diam(M)$, therefore $Z$ is closed.
Obviously $Z\ne\emptyset$, and $Z\ne S_pM$ since $\dot\gamma(0)\notin Z$.
Hence the topological boundary of $Z$ in $S_pM$ is nonempty.
Let $v\in Z$ be a vector from this boundary.
Then $\gamma_v$ is tangent to $\pd M$ at its endpoint
$q=\gamma_v(\ell(v))$. Extend $\gamma_v$ backwards until
it hits the boundary at a point $s\in\pd M$.
(By Lemma \ref{l-bg-minimizing},
the backward extension cannot have infinite length
since it starts at $q\in\pd M$.)

Since $\overrightarrow{qs}$ is tangent to the boundary,
Lemma \ref{l-dbd} implies that
$\|d_q bd_M(s,\cdot)\|=1$ where the norm is
taken with respect to the metric of~$\pd M$. Since
$bd_D$ is $C^1$-close to $bd_M$ and the metric tensors
of $M$ and $D$ are $C^0$-close at the boundary,
it follows that $\|d_q bd_D(s,\cdot)\|\approx 1$ where
the norm is taken with respect to the Riemannian metric
on $\pd D$ induced from $\R^n$.
Applying Lemma \ref{l-dbd} to $D$ yields that the
straight line segment $[qs]$ forms almost zero angle
with $\pd D$ and hence $|q-s|<\ep(\delta)$. Thus
$$
 d_M(q,s)\approx d_D(q,s)\approx 0 .
$$
Since $p$ lies on a shortest path from $q$ to $s$ in $M$,
this implies that
$$
\dist_M(p,\pd M)\le d_M(q,p)<d_M(q,s)<\ep(\delta)
$$
and the lemma follows.
\end{proof}

We are going to show that $M$ satisfies the assumptions
of Theorem \ref{t-C0} for $\ep(\delta)$ in place of~$\delta$.
The first assumption in Theorem \ref{t-C0} is
satisfied trivially.

Denote the value $\ep(\delta)$ from Lemma \ref{l-no-infinite-geodesics}
by $\rho$ and let $M'=M\setminus U_\rho(\pd M)$.
Lemma \ref{l-no-infinite-geodesics} implies that all geodesics
in $M'$ are minimizing. Hence the injectivity radius at every
point $x\in M'$ is no less than $\dist_M(x,\pd M')$. This fact and
Proposition \ref{p-isoembolic} imply that
$$
\vol(B_r(x)) \ge c(n) r^n
$$
for all $x\in M'$ and $r\le\dist_M(x,\pd M')$.
If $r\ge 2\rho$ and $B_r(x)\cap\pd M=\emptyset$,
we have $x\in M'$ and $\dist(x,\pd M')\ge r/2$, hence
$$
 \vol(B_r(x)) \ge \vol(B_{r/2}(x)) \ge c(n) (r/2)^n = 2^{-n}c(n) r^n.
$$
Thus $M$ satisfies the third requirement of Theorem \ref{t-C0}
for $2\rho$ in place of $\delta$ and $\lambda=2^{-n}c(n)$.

In order to estimate the volume of $M'$ we use
Santal\'o's formula (Proposition \ref{p-santalo}).
Let $V_M\subset SM$ be the set of all unit tangent vectors $v$
such that the geodesic $\gamma_{-v}$ eventually hits $\pd M$.
Applying Proposition \ref{p-santalo} to the set
$A=\bigcup_{p\in\pd M} S^+_pM$ yields
\be\label{e-santalo}
 \mu_L(V_M) = \int_{\pd M} d\vol_{\pd M}(p)
 \int_{S^+_pM} \ell(v) \cos\angle(v,\nu(p))\, d\vol_{S_pM}(v)
\ee
where $\nu(p)$ is the inner normal to $\pd M$ at $p$.
Let us compare the inner integral (for a fixed $p\in\pd M$)
with the similar integral for~$D$.
Let $I\colon T_pM\to T_pD$ be a linear isometry which preserves
the tangent space to the boundary and is $\ep(\delta)$-close
to the identity on it. (Such a map exists since
the metric tensors of $M$ and $D$ are close to each other on $\pd D$.)
Let $v\in S^+_pM\setminus T_p\pd M$ and $q=\gamma_v(\ell(v))$.
By Lemma \ref{l-dbd}, the $\pd M$-gradient of $-bd_M(q,\cdot)$ at $p$
equals the projection of $v$ to $T_p\pd M$.
Let $v'=I(v)\in S^+_pD\setminus T_p\pd D$
and $q'$ be the point where the ray
$\{p+tv':t>0\}$ intersects $\pd D$.
Note that $v$ and $v'$ form the same angle with the
inner normals to $\pd M$ and $\pd D$ respectively.
By Lemma \ref{l-dbd} applied to $D$, the
$\pd D$-gradient of $-bd_D(q',\cdot)$ at $p$
equals the (Euclidean) projection of $v'$ to $T_p\pd D$.
By the choice of $I$, this projection is close to
the above projection of~$v$,
therefore $d_pbd_M(q,\cdot)\approx d_pbd_D(q',\cdot)$.
Since $bd_M$ is $C^1$-close to $bd_D$, we also have
$d_pbd_M(q,\cdot)\approx d_pbd_D(q,\cdot)$.

Thus $d_pbd_D(q,\cdot)\approx d_pbd_D(q',\cdot)$.
These two derivatives determine the Euclidean directions
from $p$ to $q$ and $q'$ by means of Lemma \ref{l-dbd}.
This implies that $q\approx q'$ and therefore
$$
 \ell_M(v) = d_M(p,q)\approx |p-q| \approx |p-q'| = \ell_D(v') .
$$
This and \eqref{e-santalo} imply that $\mu_L(V_M)\approx\mu_L(V_D)$.
Observe that
$$
 \mu_L(V_D) = \omega_{n-1}\vol(D)
$$
where $\omega_{n-1}$ is the volume of the unit sphere in $\R^n$,
and
$$
 \omega_{n-1} \vol(M') \le \mu_L(V_M) \le \omega_{n-1} \vol(M)
$$
since $SM'\subset V_M\subset SM$. It follows that
$$
 \vol(M') \le \frac{\mu_L(V_M)}{\omega_{n-1}} \approx \frac{\mu_L(V_D)}{\omega_{n-1}} =\vol(D),
$$
thus $M'$ satisfies the second requirement of Theorem \ref{t-C0}
with $\ep(\delta)$ in place of $\delta$.

Thus $M$ satisfies the three requirements of Theorem \ref{t-C0}
for $\ep(\delta)$ in place of $\delta$ and some $\lambda$
depending only on~$n$. Hence $d_{GH}(M,D)<\ep(\ep(\delta))=\ep(\delta)$
by Theorem \ref{t-C0}.
This completes the proof of Theorem \ref{t-C1}.

\section{Concluding remarks and open questions}

\subsection{}
Combining the proof in this paper with technique from
\cite{BI10} and \cite{BI-new}, one can generalize the theorems to the
case when $D$ is  a region in $\mathbb H^n$,
or, more generally, a region with a Riemannian metric $C^3$-close
to the Euclidean or the hyperbolic one. To prove these generalizations,
replace the map $\phi$ in section \ref{sec-coord}
and the map $f$ in the proof of Lemma \ref{l-aux1} by
area-contracting maps constructed in \cite{BI-new}.
(The construction in \cite{BI-new} is in many ways similar to the one
in the proof of \eqref{e-phiv}; however it uses
an auxiliary map to $L^\infty(S^{n-1})$ rather than $\R^{2n}$.)

\subsection{}
It is interesting whether one can remove the assumption that $D$ is convex.
Convexity of $D$ is used in section \ref{sec-c0proof} and
in Lemma \ref{l-no-infinite-geodesics}.
The former seems easy to work around but the latter presents
more of a problem. Estimating the total volume by means
of Santal\'o's formula would not work if a significant
portion of the unit tangent bundle is covered by geodesics
that never hit the boundary. On the other hand, typical examples
where such geodesics are present have non-smooth boundary
distance functions. This raises the following question.

\begin{question}
Let $M$ be a compact Riemannian manifold with nonempty boundary
whose boundary distance function is differentiable away from the
diagonal.
Is it true that $M$ is non-trapping (that is, there are no geodesics
of infinite length)?

By Lemma \ref{l-bg-minimizing}, an affirmative answer
would imply that all geodesics in $M$ are minimizing.
Then one could ask whether the same is true for all locally
minimizing curves (i.e.\ geodesics of the length metric
rather than Riemannian geodesics).
\end{question}

\subsection{}
Another interesting question is whether
the third assumption in Theorem \ref{t-C0}
can be replaced by the following: every metric ball
of radius $r$ in $M$ (sufficiently separated away from the boundary)
is contractible within a ball of radius $\rho(r)$
where $\rho:[0,+\infty)\to[0,+\infty)$
is a fixed function such that $\rho(r)\to 0$ as $r\to 0$.
As shown in \cite{GP}, volumes of $r$-balls (separated
away from the boundary) in such $M$ are uniformly bounded below by
$\nu=\nu_{\rho}(r)>0$. This is similar to the third assumption
of Theorem \ref{t-C0} except that $\nu_\rho(r)$ is not of
the form $\lambda r^n$. (This form of a volume bound is
used in Lemma \ref{l-phi-almost-injective}.)

It is easy to see that a class of Riemannian manifold with a
uniform lower bound on volumes of balls depending only on radius,
and uniformly bounded diameter and total volume, is pre-compact
in Gromov--Hausdorff topology. Therefore a sequence of manifolds
satisfying assumptions (1) and (2) of Theorem \ref{t-C0} and
the above uniform local contractibility assumption, must have
a partial Gromov--Hausdorff limit. One could try to equip this
limit with a structure allowing one to analyze the equality case
in Besikovitch inequality. Such a structure would certainly have
applications beyond this particular question.

\end{document}